\documentclass[10pt,a4paper,reqno,oneside]{amsart}

\usepackage{amsrefs}

\usepackage{array}
\usepackage{url}
\usepackage{multicol}
\usepackage[linktocpage = true]{hyperref}
\usepackage{todonotes}

\hypersetup{
 colorlinks = true,
 citecolor = blue,
}
\usepackage{color}
\definecolor{red}{rgb}{1,0,0}
\definecolor{green}{rgb}{0,1,0}
\definecolor{blue}{rgb}{0,0,1}
\definecolor{refkey}{gray}{.625}
\definecolor{labelkey}{gray}{.625}
\usepackage[a4paper]{geometry}
\usepackage{amsmath,amssymb}
\usepackage{amssymb}
\usepackage{amsmath}
\usepackage{amsthm}
\usepackage{enumerate}
\usepackage{graphicx}

\usepackage{tikz-cd}
\allowdisplaybreaks

\def\id{\mathrm{Id}}

\makeatletter

\newcommand{\Rmnum}[1]{\expandafter\@slowromancap\romannumeral #1@}
\makeatother

\theoremstyle{plain}
\newtheorem{thm}{\protect\theoremname}[section]
\newtheorem{prop}[thm]{\protect\propositionname}
\newtheorem{cor}[thm]{\protect\corollaryname}
\newtheorem{lem}[thm]{\protect\lemmaname}

\theoremstyle{definition}

\newtheorem{defn}[thm]{\protect\definitionname}
\newtheorem{remark}[thm]{\protect\remarkname}
\newtheorem{notation}[thm]{\protect\notationname}

\AtBeginDocument{
	
}

\makeatother

  \providecommand{\corollaryname}{Corollary}
  \providecommand{\examplename}{Example}
  \providecommand{\lemmaname}{Lemma}
  \providecommand{\propositionname}{Proposition}
  \providecommand{\theoremname}{Theorem}
  \providecommand{\definitionname}{Definition}
  \providecommand{\remarkname}{Remark}
  \providecommand{\notationname}{Notation}

  \makeatletter
  \newcommand{\be}{%
  \begingroup
  \eqnarray%
   \@ifstar{\nonumber}{}%
  }

\newcommand{\newF}{\tilde{\mathcal{F}}}

\newcommand{\F}{\mathcal{F}}

\newcommand{\C}{\mathcal{C}}

\newcommand{\Aut}{\operatorname{Aut}}
\newcommand{\End}{\operatorname{End}}

\makeatother

\begin{document}

\title[A Note on cyclotomic function fields with quadratic modulus]{A Note on cyclotomic function fields with quadratic modulus}
%

\author{Haojie Chen}
\address{Sun Yat-Sen University, School of Mathematical, Guangzhou, China}
\email{\href{chenhj69@mail2.sysu.edu.cn}{chenhj69@mail2.sysu.edu.cn}}

\author{Chuangqiang Hu}
\address{Sun Yat-Sen University, School of Mathematical, Guangzhou, China}
\email{\href{huchq@bimsa.cn}{huchq@bimsa.cn}}
%

\allowdisplaybreaks

\begin{abstract}
A longstanding and important problem in algebraic geometry is the characterization of  algebraic function fields. In this paper, we focus on the characterization problem for cyclotomic function field $L(\Lambda_M)$, which is an important class of explicit function fields with applications in number theory and coding theory. Motivated by Arakelian and Quoos' classification of $L(\Lambda_M)$ with an irreducible quadratic modulus, we provide a complete characterization of the cyclotomic function field $L(\Lambda_M)$ with modulus $M = x^2$. More precisely, we prove that a function field $\mathcal{F}$ over $\mathbb{F}_q$ is $\mathbb{F}_q$-isomorphic to $L(\Lambda_{x^2})$ if and only if it satisfies the following three conditions: (i) $\mathcal{F}$ has a subgroup $G$ isomorphic to the direct product $(\mathbb{F}_q,+) \times \mathbb{F}_q^*$; (ii) its genus is $g(\mathcal{F}) = 1 + q(q-3)/2$; and (iii) the cardinality of $\mathbb{F}_q$-rational places is exactly $q+1$.
\end{abstract}
\maketitle{}

%
%







\section{Introduction}
The study of birational invariants such as genus, automorphism group, and number of rational places plays a fundamental role in the classification of algebraic function fields. The simplest example is the case of Hermitian function field and its subfields, in which we obtain a complete characterization via the two invariants: the number of rational places and the genus. Recall that a function field $F$ defined over the finite field $\mathbb{F}_{q^2}$ is called maximal if the number of rational places $N(F)$ attains the Hasse-Weil bound, i.e., 
 \[ N( F ) = q^2+1+2gq , \]
  where $g$ is the genus of $F$. It is well known that the Hermitian function field, defined by the equation 
  \[ y^q + y = x^{q+1} \]
   has genus $q(q-1)/2$ with $N=q^4+1+2 g q^2$ rational places. As shown in \cite{MR1305281}, the Hermitian function field is the unique maximal function field over $\mathbb{F}_{q^2}$ with genus $q(q-1)/2$, up to $\mathbb{F}_{q^2}$-isomorphism. In fact, it is the largest possible genus for a maximal function field over $\mathbb{F}_{q^2}$. 
The second largest genus for a maximal function field over $\mathbb{F}_{q^2}$ is $\left\lfloor\frac{(q-1)^2}{4}\right\rfloor$. The examples of such function fields can be obtained as subfields of the Hermitian function field, namely $y^{\frac{q+1}{2}}=x^q+x$ for odd $q$ (see \cite{MR1485426}) and $y^{q+1}=x^{q/2}+x^{q/4}+\cdots+x^2+x$ for even $q$ (see \cite{MR4329278}).
In \cite{MR1485426} and \cite{MR4329278} it was proven that the two function fields above are the only maximal function fields, up to $\mathbb{F}_{q^2}$-isomorphism, of genus $\left\lfloor\frac{(q-1)^2}{4}\right\rfloor$.

Conversely, there are also examples that indicate that the genus and the rational places are not enough to classify the function fields. In \cite{MR2257083}, Giulietti et al. classified the family of maximal function fields $\mathbb{F}_{q^2}(x, y)$ of genus $q-1$ over $\mathbb{F}_{q^2}$ defined by
$$
y^{q+1}=x^{2 i}\left(x^2+1\right),
$$
under the conditions that $\frac{q+1}{2}>3$ is a prime number and $1 \leqslant i \leqslant\frac{q-3}{2}$. They also proved that this family gives rise to roughly $(q+1) / 12$ non-isomorphic maximal function fields.

Besides, some interesting examples yield that some algebraic function fields can be characterized by its genus and a certain subgroup of its automorphism group. For instance, the work in \cite{MR3339574} showed that the Artin-Mumford function field is the unique function field (up to $\mathbb{F}_p$-isomorphism) over $\mathbb{F}_p$ of genus $(p-1)^2$ whose automorphism group contains a subgroup isomorphic to $\left(\mathbb{Z}_p \times \mathbb{Z}_p\right) \rtimes D_{p-1}$.
Similarly, some function fields can be characterized by specific automorphism subgroups together with the structure of their fixed fields, as illustrated in \cite{MR3670362}.

Some classification results require the combination of the three invariants above.
When $ q = n^3 $, the authors of \cite{MR2448446} showed that the GK function field is the unique $\mathbb{F}_{q^2}$-maximal function field, up to $\mathbb{F}_{q^2}$-isomorphism, of genus  
$g = \frac{1}{2}(n^3 + 1)(n^2 - 2) + 1$
with an automorphism group over $\mathbb{F}_{q^2}$ of order $n^3(n^3 + 1)(n^2 - 1)(n^2 - n + 1)$.  
Remarkably, in some cases, three invariants are insufficient to characterize all kinds of function fields. The work in \cite{MR4843058} showed examples of maximal function fields that have the same genus and automorphism group but are not isomorphic. Additionally, some classification results require the algebraic structure of the Weierstrass semigroup. For example, it was shown in \cite{MR3944454} that the so-called Ree function field is unique, up to $\mathbb{F}_q$-isomorphism, given its number of rational places, its genus, and the shape of two elements of the Weierstrass semigroup at a rational place.

In this paper, we try to derive a similar classification result as in \cite{arakelian2024cyclotomicfunctionfieldsfinite} concerning the family of cyclotomic function fields. Cyclotomic function fields play an important role in class field theory of function fields. Applying the theory of cyclotomic function fields,  the authors of \cite{MR4433231} constructed binary sequences with low correlation. Moreover, \cite{MR3734181} gave a construction of sequences with high nonlinear complexity. For other interesting applications of cyclotomic function fields, we refer to \cite{MR4116821,MR4017935,MR4195891} and the references therein.

It is natural to ask whether a cyclotomic function field can be characterized by its genus, number of rational places, and a specific subgroup of its automorphism group. In \cite{arakelian2024cyclotomicfunctionfieldsfinite}, Arakelian and Quoos partially answered this question for the case where $M$ is an irreducible quadratic modulus. They showed that a function field over $\mathbb{F}_q$ with genus $\frac{(q+1)(q-2)}{2}$, having $q+1$ rational places, and possessing a subgroup of automorphisms isomorphic to $\mathbb{F}_{q^2}^*$, must be $\mathbb{F}_q$-isomorphic to a cyclotomic function field of this type. Inspired by their work, this paper aims to prove a similar characterization for the cyclotomic function field $L(\Lambda_M)$ with $M = x^2$, using its genus, number of rational places, and a certain automorphism subgroup.

This paper is outlined as follows.
In Section \ref{sec:preliminaries}, we introduce notation and some results concerning function fields, automorphism groups, cyclotomic function fields and Kummer extensions. 
In Section \ref{sec:mainresult}, we give the complete characterization of $L(\Lambda_M)$ with $M=x^2$ by a key claim about the short orbits. Finally, in Section \ref{sec:proofofLemma}, we provide a proof of the claim.
\section{Preliminaries}\label{sec:preliminaries}
 In this section, we recall some notation and  fundamental results concerning the theory of automorphism groups of algebraic function fields and cyclotomic function fields. Throughout, let $\mathbb{F}_q$ denote the finite field with $q$ elements, where $q$ is a power of a prime integer $p$, and let $\overline{\mathbb{F}}_q$ denote its algebraic closure.

\subsection{Hurwitz genus formula}\label{sub:funtionfields}
Let $F / K$ be an algebraic function field of genus $g$ with constant field $K$ and let $F' / F$ be a finite separable extension. Let $K^{\prime}$ denote the constant field of $F'$ and $g^{\prime}$ the genus of $F' / K^{\prime}$. 
Consider a place $P$ of $F$ and a place $P'$ of $F'$ lying over $P$. Let $d\left(P^{\prime} \mid P\right)$ be the different exponent of $P^{\prime}$ lying over $P$.
Denote the set of places of $F$ by $\mathbb{P}_F$. We define the different of $F' / F$ as  
$$
\operatorname{Diff}\left(F' / F\right):=\sum_{P \in \mathbb{P}_F} \sum_{P^{\prime} \mid P} d\left(P^{\prime} \mid P\right) \cdot P^{\prime} .
$$
Then we have the famous Hurwitz genus formula:
\begin{thm}\cite{MR2464941}\label{thm:genus}
    With the notation above, the genus $ g' $ of $F'$ is given by the formula
$$
2 g^{\prime}-2=\frac{\left[F': F\right]}{\left[K^{\prime}: K\right]}(2 g-2)+\operatorname{deg} \operatorname{Diff}\left(F' / F\right).
$$
\end{thm}
Let $e\left(P^{\prime} \mid P\right)$ denote the ramification index of $P'$ lying over $P$. The following theorem establishes a close relationship between the ramification index and the different exponent.
\begin{thm}\label{thm:Dedekind's Different Theorem}\cite{MR2464941}
    With notation as above, we have
    \begin{enumerate}
        \item  $d\left(P^{\prime} \mid P\right) \geqslant e\left(P^{\prime} \mid P\right)-1$.
         \item  $d\left(P^{\prime} \mid P\right)\geqslant  e\left(P^{\prime} \mid P\right)$ if and only if $e\left(P^{\prime} \mid P\right)$ is divisible by $\operatorname{char} K$.
    \end{enumerate}
\end{thm}
Let $P$ be a place of $F / K$ and let $P_1', \ldots, P_m'$ be all the places of $F' / K^{\prime}$ lying over $P$. Let $e_i:=e\left(P_i' \mid P\right)$ denote the ramification index and $f_i:=f\left(P_i' \mid P\right)$ the relative degree of $P_i' \mid P$. Then the Fundamental Equality states that 
\begin{equation}\label{eq:FundamentalEquality}
    \sum_{i=1}^m e_i f_i=\left[F': F\right].
\end{equation}
\subsection{Automorphism groups of function fields}
Now we recall some basic results of the  automorphism group  of function fields. 
Let $ F $ be an algebraic function field over $ K $, and let $G$ be a subgroup of its automorphism group $\Aut_{ K }( F )$. Obviously, the group $ G $ can be viewed as an action on the set of places of $ F $. For a place $ P $ of $ F $, we denote by 
\[
    G_{ P }= \{h \in G: h( P )=P \}
\]
the stabilizer of $ P $, and by
\[
    G( P )=\{h( P): h \in G\}
\]
the orbit of $  P $ under the action of $G$. We refer to $G( P )$ as a $G$-orbit of $ F $.  
 \begin{defn}\label{def:shortOrbit}
    A $G$-orbit of $F $ is called a short $G$-orbit of $F$, if its cardinality is less than $|G|$.
\end{defn}
Denote by $ F ^G$ the fixed field of $ F $ under $G$. From Galois theory, the extension $ F/ F^G$ is known to be a Galois extension with Galois group $G $.
Suppose that $ P $ is a place lying over $ Q $ in the field extension  $ F/ F^G$. 
\begin{lem}\cite{MR2464941}
Let $ e(P  \mid Q ) $ and $f (P \mid Q ) $ denote the ramification index and relative degree as above.  
Then 
\begin{equation}\label{eq:G_P}
     | G_{ P } | = e(P  \mid Q ) f (P \mid Q )
\end{equation}
and 
\begin{equation}\label{eq:GP}
   | G( P )|  = \frac{ |G| }{e(P  \mid Q ) f (P \mid Q )}. 
\end{equation}
\end{lem}
\begin{remark}
If we assume that the constant field $ K $ is algebraically closed, then 
$f (P \mid Q ) = 1 $. In this situation, Equation \eqref{eq:GP} yields that the $G$-orbit is short if and only if $ e(P  \mid Q ) > 1 $.
\end{remark}

\subsection{Hurwitz Genus Formula in terms of short orbits}
Now we assume that $\tilde{F}  $ is a function field over an algebraically closed field $ \bar{\mathbb{F}}_q $. Let $\tilde{F}^G$ be the fixed field of $\tilde{F}$ under the action of a subgroup $G$ of $\Aut_{\bar{\mathbb{F}}_q}(\tilde{F})$.
Let $g$ and $g^G$ denote the genus of $\tilde{F}$ and $\tilde{F}^G$, respectively. Let $\Omega_1, \ldots, \Omega_k$ be all the short $G$-orbits and let $l_1, \ldots, l_k$ be the cardinalities of the short $G$-orbits. Denote by $e_i$ the ramification index of $\Omega_i$ and $d_i$ the different exponent of $\Omega_i$. This is well defined, since $\tilde{F} /\tilde{F}^G $ is a Galois extension. In terms of short orbits, the Hurwitz Genus Formula (Theorem \ref{thm:genus}) can be rewritten as follows:

\begin{align}
    2 g-2 &= |G|\left(2 g^G-2\right)+ \sum_{i=1}^k d_i l_i \nonumber \\
    &\geqslant |G|\left(2 g^G-2\right)+\sum_{  1 \leqslant i\leqslant k , p \mid e_i} e_i l_i+\sum_{ 1 \leqslant i\leqslant k  ,p \nmid e_i } (e_i-1) l_i  \label{eq:HurWitzgenusFormula11}\\
    & \geqslant |G|\left(2 g^G-2\right)+\sum_{  1 \leqslant i\leqslant k   } (e_i-1)l_i\label{eq:HurWitzgenusFormula12}.
\end{align}
In particular, if $\operatorname{gcd}(p,|G|)=1$, then 
\begin{equation}\label{eq:HurWitzgenusformula2}
    2 g-2 = |G|\left(2 g^G-2\right)+\sum_{i=1}^k\left( e_i-1 \right)l_i =|G|\left(2 g^G-2\right)+\sum_{i=1}^k\left(|G|-l_i\right) .
\end{equation}

\subsection{Cyclotomic function fields}\label{sub:cyclotomicfunctionfield}
We provide a concise overview of the theory of cyclotomic function fields. All results presented in this subsection can be found in \cite{MR2241963}.

Let $L = \mathbb{F}_{q}(x)$ be the fraction field of polynomial ring $\mathbb{F}_q[x]$ over $ \mathbb{F}_q$. Denote by $\bar{L}$ the algebraic closure of $L$. 
A map $ \varphi : \bar{L} \to \bar{L} $ is called an $ \mathbb{F}_q $-linear map if \[
\varphi(a+b)=\varphi(a)+\varphi(b) \]
and
\[
\varphi(\alpha a)=\alpha \varphi(a)  \]
hold for $ a,b\in \bar{L} $ and $ \alpha \in \mathbb{F}_q $.
Let $\End_{\mathbb{F}_{q}}(\bar{L})$ denote the set of $ \mathbb{F}_q $-linear maps. Equipped with the composition map, $\End_{\mathbb{F}_{q}}(\bar{L})$ is a non-commutative $\mathbb{F}_q$-algebra.

\begin{defn}[Carlitz Module]
Consider a specific $ \mathbb{F}_q $-linear map $ \C_x \in \End_{\mathbb{F}_{q}}(\bar{L})$ given by
$$
\C_x (u)=u^q+x u, \quad u \in \bar{L}.
$$
Denote by $\C_x^{(k)} $  the $k $-th composition of $ \C_x $.
Given any $f(x) \in \mathbb{F}_q[x]$, the substitution $x^k  \rightarrow \C_x^{(k)} $ in $f$ gives an element, say $ \C_f $, of $\End_{\mathbb{F}_{q}}(\bar{L})$. 
Precisely, if $ f(x)=a_n x^n+\cdots+a_1 x+a_0 $,  then
$$
\C_f (u)=a_n \C_x^{(n)}(u)+\cdots+a_1 \C_x^{(1)}(u)+a_0 u 
$$
for all $u \in \bar{L}$.
Thus, we obtain a ring homomorphism 
\begin{align*}
  \mathbb{F}_q[x] & \rightarrow \End_{\mathbb{F}_{q}}(\bar{L}) \\
f  & \mapsto \C_f ,
\end{align*}
 which is usually called the Carlitz module over $\bar{L}$.
\end{defn}
Moreover, the homomorphism $\C$ equips $\bar{L}$ with an $\mathbb{F}_q[x]$-module structure.
If $u \in \bar{L}$ and $M \in \mathbb{F}_q[x]$, we write $u^M=\C_M(u)$.  
For $M, N \in \mathbb{F}_q[x]$, it is evident that
$$
u^{M+N}=u^M+u^N \quad \text { and } \quad u^{M N}=\left(u^M\right)^N.
$$
Suppose that $M$ is a non-zero polynomial. Then the set of $M$-torsion points
\[
\Lambda_M=\left\{u \in \bar{L} \mid u^M=0\right\}
\]
forms a finite $\mathbb{F}_q[x]$-submodule of $\bar{L}$. In fact, it is well-known that $\left|\Lambda_M\right|=q^{\operatorname{deg} M}$ and $\Lambda_M \cong \mathbb{F}_q[x] /(M)$. 
\begin{defn}[Cyclotomic Function Field]
The cyclotomic function field with modulus $M$, denoted by $L(\Lambda_M)$, is defined as the subfield of $\bar{L}$ generated over $L$ by the elements of $\Lambda_M$.
\end{defn}
It is well-known that $L\left(\Lambda_M\right)$ is a Galois extension of $L$ with Galois group
\[
\operatorname{Gal}\left(L\left(\Lambda_M\right) / L\right) \cong(\mathbb{F}_q[x] /(M))^*,
\]
where $(\mathbb{F}_q[x] /(M))^*$ is the unit group of $\mathbb{F}_q[x] /(M)$. 
In particular, we have the following results in the case $M=x^{n+1}$ where $ n \geqslant 1 $ (see \cite{MR3531248} and \cite{MR2241963}). 
\begin{thm}
Let $ L(\Lambda_M) $ be the cyclotomic function field with modulus $ M=x^{n+1} $ where $ n \geqslant 1 $. Then
\begin{enumerate}
    \item $\Aut_{\mathbb{F}_q}(L(\Lambda_M))$ has a subgroup $G$ isomorphic to $ ( \mathbb{F}_q[x]/(x^{n+1}) )^*.$
     \item $g(L(\Lambda_M))=1+\frac{q^n(nq-n-2)}{2}$.
     \item $L(\Lambda_M)$ has exactly $ q^n+1 $ rational places over $\mathbb{F}_q$.
     \item If $g(L(\Lambda_M)) \geqslant 2$, then the automorphism group of $L(\Lambda_{x^{n+1}})$ over $\mathbb{F}_q$ coincides with the Galois group
$\operatorname{Gal}\bigl(L(\Lambda_{x^{n+1}})/L\bigr) \cong (\mathbb{F}_q[x]/(x^{n+1}))^*$.
\end{enumerate}
\end{thm}
\subsection{Kummer extension}\label{sub:kunmmenextension}
Next, we give a brief introduction to the theory of Kummer extension. 
Let $K$ be a perfect field of characteristic $p \geqslant 0$. 
Assume that $K$ contains a primitive $n$-th root of unity $\zeta_n$ with $(n, p)=1$. Let $F$ be a function field over $K$ and $F'$  a finite field extension of $F$. 
\begin{thm}[Theorem 5.8.5 in \cite{MR2241963}] \label{thm:kummer_extensions}
    The field extension $F' / F $ is  cyclic of degree $n$ if and only if there exists some $y \in F'$ such that $F'=F(y)$ and the minimal polynomial of $y$ over $ F $ is given by 
   \[
   T^n-h \in F[T] 
   \]
 for some $ h \in F $. Furthermore, by applying a coordinate transformation, one can choose $h \in F $ such that $ 0 \leqslant v_{P} (h)\leqslant n-1 $ for  any place $P$ of $F$.  
\end{thm}
The field extension $F' / F$ in the theorem above is called a \textbf{Kummer extension} of degree $n$. Let $\zeta_n$ be a primitive $n$-th root of unity.
The Galois group of $F' / F$  is generated by $ \sigma  $,  where
 \begin{equation}\label{eq:automorphism_kummer}
     \sigma (y)=\zeta_n y \text{ and} \left.\sigma\right|_{F} =  \id .
    \end{equation} 
The following theorem characterizes the ramification indices of a Kummer extension.
\begin{thm}[Theorem 5.8.12 in \cite{MR2241963}]\label{thm:kummer_extension2}
 Let  $F' / F$ be a Kummer extension of degree $n$. 
Assume that $v_{P}(h)=m$ and $P^{\prime}$ is a place of $F'$ over $P$. We have
$$
e(P^{\prime} \mid P)=\frac{n}{(n, m)} .
$$
\end{thm}
 The following proposition comes from \cite{arakelian2024cyclotomicfunctionfieldsfinite}, which is useful for the proof of our main theorem. 
\begin{prop}\cite{arakelian2024cyclotomicfunctionfieldsfinite}\label{prop:2.4}
   Let $F'=F(y)$ be a Kummer extension of $F$ defined over $K$ by $y^n=h$ as before. Let $ \sigma  $ be the automorphism \eqref{eq:automorphism_kummer} generating the Galois group $\operatorname{Gal}(F'/F)$.
    Then the following statements hold:
    \begin{enumerate}
        \item Assume that $\tau \in \Aut_K(F' )$ is a nontrivial automorphism that normalizes the Galois group, i.e.,
     $\tau \langle \sigma \rangle \tau^{-1}=\langle \sigma \rangle$.  Then there exists a unique index $k \in\{1, \ldots, n-1\}$ with $k \nmid n$ such that 
    \[  \tau(y)=f y^k , \]
     where $f \in F$ satisfies $f^n=\tau(h) / h^k$. 
     \item If additionally $\tau$ commutes with $ \sigma $, then $k =1  $.
    \end{enumerate}
\end{prop}
\begin{proof}
    \begin{enumerate}
        \item  Since $ 1, y , \cdots, y^{n-1} $ form a basis of $ F'/ F $, the automorphism $ \tau $ can be written as 
\[ \tau(y)=\sum_{i=0}^{n-1} f_i y^i \]
for  some $f_i \in F$.  Since $\tau$ is nontrivial, we may assume that $ f_k \not = 0 $ for some $k$.  If $\tau$ normalizes $\langle \sigma \rangle$, then we write 
\[ \sigma^{l} \tau=\tau \sigma 
\]
  for some $ l \in \{ 1,\ldots,n-1 \} $. Thus
  \[ 
\sum_{i=0}^{n-1}\left(\zeta_n^{i l}-\zeta_n\right) f_i y^i=0.
\]
We conclude that $ \left(\zeta_n^{i l}-\zeta_n\right) f_i =0 $ for each $i $. Then  $k l \equiv 1 \bmod n$, and $f_i=0$ for $i \neq k$. So $\tau$ is given by $\tau(y)=f_k y^k$. Moreover, from $  y^n =  h $, we obtain 
$$
\tau(h) = \tau\left(y^n\right)=(\tau(y))^n=f_k^n\left(y^n\right)^k=f_k^n h^k .
$$
That is $f_k^n=\tau(h) / h^k$.
\item 
In particular, if $\tau$ commutes with $\sigma$, i.e., $l = 1$, then the relation $k l \equiv 1 \pmod{n}$ yields $k = 1$.
    \end{enumerate}   
\end{proof}

\section{Main results}\label{sec:mainresult}
In what follows,  we let $\mathcal{F}$ denote a function field over $\mathbb{F}_q$ satisfying the conditions:
\begin{enumerate}[(A)]\label{condition123}
     \item $\mathcal{F}$ has a subgroup $G$ isomorphic to $(\mathbb{F}_q,+)   \times \mathbb{F}_q^*$, where $(\mathbb{F}_q,+)$ denotes the additive group of the finite field $\mathbb{F}_q$. \label{con:subgroup}
     \item $g(\mathcal{F})=1+\frac{q(q-3)}{2}$.\label{con:genus}
     \item $\mathcal{F}$ has exactly $q+1$ $\mathbb{F}_q$-rational places.\label{con:rationalpoints}
\end{enumerate}
From Section \ref{sub:cyclotomicfunctionfield} and the isomorphism 
\[ ( \mathbb{F}_q[x]/(x^2) )^* \cong (\mathbb{F}_q,+)   \times \mathbb{F}_q^* ,
 \]
the cyclotomic function field $L(\Lambda_M)$ is a typical example satisfying the three conditions above.

Our main result states that the converse is true. 
\begin{thm}\label{thm:main}
    Let $\mathcal{F}$ be a function field of genus $g$ over $\mathbb{F}_q$. Assume that $\mathcal{F}$ verifies the conditions $\eqref{con:subgroup}\eqref{con:genus}\eqref{con:rationalpoints}$. 
Then $\mathcal{F}$ is $\mathbb{F}_q$-isomorphic to the cyclotomic function field $L\left(\Lambda_{x^2}\right)$.
\end{thm}
When $q = 2$, the field $\mathcal{F}$ is simply the rational function field, and the theorem above holds trivially. It therefore suffices to consider the case $q \geqslant 3$.

Firstly, we give another expression for $L(\Lambda_{x^2})$ as follows.
\begin{prop}\label{kummerformoftheT^2}
    The cyclotomic function field $L(\Lambda_{x^2})$ with modulus $x^2$ is $\mathbb{F}_q$-isomorphic to the function field $\mathbb{F}_q(u,v)$ defined by 
    \begin{equation}\label{eq:kummerformoftheT^2}
        u^{q-1}=\lambda (v^q-v).
    \end{equation}
    for some constant $\lambda \in \mathbb{F}_q^{*}$.
\end{prop}
\begin{proof}
    From the construction of cyclotomic function fields, we know that $L(\Lambda_{x^2})=\mathbb{F}_q(x,y)$, with 
    \[
        (y^q+xy)^{q-1}+x=0.
    \]
    Define the new variables 
    \[ u=\frac{1}{y^q+xy},v=\frac{y}{\lambda (y^q+xy)}.\]
   It is evident that $y=\lambda \frac{v}{u}$ and $x=-\frac{\lambda}{v^q-v}$. So we have $L(x,y) = L(u,v)$.
    The proposition follows by checking that $  u  $  and $ v $ satisfy the equation \eqref{eq:kummerformoftheT^2}.
\end{proof}
\begin{notation}\label{not:subgroups}
Since $G$ is isomorphic to $ (\mathbb{F}_q,+)   \times \mathbb{F}_q^*$, it follows that $G$ contains a unique subgroup of order $q-1$ and a unique subgroup of order $q$. Denote such groups by $H$ and $I$, respectively.
\end{notation}
Our main technique for proving Theorem \ref{thm:main} is to analyze the ramification structure of the extension $\mathcal{F}/\mathcal{F}^H$ by using the Hurwitz genus formula. For this purpose, we need to understand the $G$-orbits of $\mathcal{F}$, which is given by the following lemma. 
\begin{lem}\label{lem:orbits}
   Let $\mathcal{F}$ be the function field satisfying conditions \eqref{con:subgroup}, \eqref{con:genus}, and \eqref{con:rationalpoints}. Then the action of $G$ on $\mathcal{F}$ splits the set of rational places of $\mathcal{F}$ into two short $G$-orbits, $\Omega_1$ and $\Omega_2$, with $|\Omega_1| = q$ and $|\Omega_2| = 1$. Moreover, the collection of all short $G$-orbits of $\mathcal{F}$ consists exactly of $\Omega_1$ and $\Omega_2$.
\end{lem}
We shall postpone the proof of this technical lemma to the next section. 
As a consequence of Lemma \ref{lem:orbits}, we derive that the ramification places in the extension $\mathcal{F}/\mathcal{F}^H$ are exactly the rational places of $\mathcal{F}$. Moreover, they are all totally ramified as shown in the following lemma.
\begin{lem}\label{lem:finalproof0}
    Let $\mathcal{F}$ be a function field satisfying \eqref{con:subgroup}\eqref{con:genus}\eqref{con:rationalpoints}. Let $H$ be the unique subgroup of $G$ of order $q-1$. Then all rational places of $\mathcal{F}$ are totally ramified in the extension $\mathcal{F}/\mathcal{F}^H$.
\end{lem}
\begin{proof}
    By the first statement of Lemma \ref{lem:orbits}, the rational places of $\mathcal{F}$ under the action of $G$ form exactly two short $G$-orbits: an orbit $\Omega_1$ of cardinality $q$ and an orbit $\Omega_2$ of cardinality $1$. So the ramified places in the extension $\mathcal{F}/\mathcal{F}^G$ are exactly the rational places of $\mathcal{F}$. Since $\mathcal{F}^H $ is the fixed field of $H$ and $H$ is a subgroup of $G$, the ramified places in the extension $\mathcal{F}/\mathcal{F}^H$ are also exactly the rational places of $\mathcal{F}$. 
    
     Now we prove that these places are all totally ramified in the extension $\mathcal{F}/\mathcal{F}^H$. From Equation \eqref{eq:GP}, 
      it suffices to show that $H$ fixes all places in the short $G$-orbits.
      
      As a subgroup of $G$, $H$ clearly fixes the orbit $\Omega_2$, i.e., the unique place in $\Omega_2$ is fixed by $H$. 
    Take a place $P \in \Omega_1$. The stabilizer of $P$ under the action of $G$ has order $\frac{|G|}{|\Omega_1|} = q-1$. Since $H$ is the unique subgroup of $G$ of order $q-1$, it follows that $H$ is exactly the stabilizer of $P$ (and hence of every place in $\Omega_1$).  
\end{proof}
The following corollary yields that the fixed field $\mathcal{F}^H$ is rational, which is crucial for the proof of Theorem \ref{thm:main}.
\begin{cor}\label{cor:genus0}
    Let $\mathcal{F}$ be the function field satisfying conditions \eqref{con:subgroup}, \eqref{con:genus}, and \eqref{con:rationalpoints}.
 Let $H$ be the unique subgroup of $G$ of order $q-1$. Then the fixed field $\mathcal{F}^H$ is a rational function field.
\end{cor}
\begin{proof}
     By Lemma \ref{lem:finalproof0}, the places of $\Omega_1$ and $\Omega_2$ are totally ramified in the field extension $\F/\F^H$. Denote by $g^H$ the genus of $\F^H$. By Hurwitz Genus Formula \eqref{eq:HurWitzgenusformula2},
    \[
        q(q-3)=(2g^H-2)(q-1)+q(q-1-1)+(q-1-1).
    \]
    This implies that $ g^H = 0 $, i.e., $ \mathcal{F}^H$ is rational.
    We complete the proof.
\end{proof}
Next, we derive a Kummer extension form of the function field $\mathcal{F}$. 
\begin{lem}\label{lem:finalproof1}
Let $\mathcal{F}$ be a function field satisfying \eqref{con:subgroup}\eqref{con:genus}\eqref{con:rationalpoints}. Then the function field  
\[
\mathcal{F} = \mathbb{F}_q(v, y)
\]  
is given by the Kummer equation  
\[
y^{q-1}= \lambda \prod_{\alpha_i \in \mathbb{F}_q}\left(v-\alpha_i\right)^{s_i},
\]
where:
 \begin{enumerate}
    \item   $ \lambda \in \mathbb{F}_q^{*} $; 
\item  Each exponent $ s_i $ is a positive integer satisfying $ 1 \leqslant s_i < q-1 $, and coprime to $ q-1 $;
\item The total sum $ S := \sum_{i } s_i $ is also coprime to $ q-1 $.
\end{enumerate}
Furthermore, we may assume that the short $G$-orbit of cardinality $1$ consists of the unique place at infinity by applying a coordinate transformation. 
 \end{lem}
\begin{proof}
Denote by $H \subset G$ the subgroup as in Notation \ref{not:subgroups}. From Corollary \ref{cor:genus0}, the fixed field $\mathcal{F}^H$ is rational. Assume that $v \in \mathcal{F}^H$ such that $\mathcal{F}^H=\mathbb{F}_q(v)$.  Notice that $ \mathcal{F} / \mathcal{F}^H$ is cyclic with Galois group $ H $. 
From Theorem \ref{thm:kummer_extensions} we have that $\mathcal{F}$ is a Kummer extension of $\mathbb{F}_q(v)$ of degree $ | H | = q-1$.

By Theorem \ref{thm:kummer_extensions}, the field $\mathcal{F}$ can be written as $\mathcal{F} = \mathbb{F}_q(v, z)$, where $z$ satisfies the Kummer equation
\begin{equation}\label{eq:kummer1}
z^{q-1} = h(v),
\end{equation}
for some rational function $h(v) \in \mathbb{F}_q(v)$. Let $\mathbb{F}_q=\{\alpha_0, \dots,\alpha_{q-1}\}$. Assume that $h(v)$ admits the factorization 
\[
h(v) = \lambda \prod_{i=0}^{q-1} (v-\alpha_i)^{r_i} \prod_{j} p_j(v)^{t_j}
\]
where $\lambda \in \mathbb{F}_q^*$, $r_i \in \mathbb{Z}$, $t_j \in \mathbb{Z}$,  and $p_j(v)$ are distinct monic irreducible polynomials of degree $> 1$.

 From Lemma \ref{lem:finalproof0}, the $q+1$ rational places of $\mathcal{F}$ lie in $\Omega_1$ and $\Omega_2$, denoted by $P_0,\dots,P_{q-1}$ (the places in $\Omega_1$) and $P_\infty$ (the place in $\Omega_2$), and they are totally ramified. These places, when restricted to $\mathbb{F}_q(v)$, must be the rational places of $\mathbb{F}_q(v)$.
It is clear that the rational places of $\mathbb{F}_q(v)$ are exactly the places corresponding to $v-\alpha_i$ for $i=0,\dots,q-1$ and the place at infinity.
We conclude that $t_j$ is divisible by $q-1$, while each $ r_i $ is coprime to $ q-1 $ by Theorem \ref{thm:kummer_extension2}.
Assume that $ r_i = n_i(q-1) + s_i $ with $ 1 \leqslant s_i < q-1 $, and that $ t_j = (q-1)m_j $ for some integers $ n_i $ and $ m_j $. Performing the change of variable 
\[ y = \frac{z}{\prod_{i=0}^{q-1} (v-\alpha_i)^{n_i} \prod_{j} p_j(v)^{m_j}} ,
\]
the equation \eqref{eq:kummer1} can be rewritten as
\[
    y^{q-1}=\lambda \prod_{\alpha_i \in \mathbb{F}_q}\left(v-\alpha_i\right)^{s_i}.
\]
Obviously,  $\mathbb{F}_q(v, z)=\mathbb{F}_q(v, y)$ and then the first assertion follows. 
The second assertion is clear by applying the transformation
\[
    v \mapsto \frac{1}{v - \alpha_i}
\]
if $P_\infty'$ lies over the zero of $v - \alpha_i$.
\end{proof}
In fact, we can further simplify the Kummer equation in Lemma \ref{lem:finalproof1} such that all $s_i$ are the same.
\begin{lem}\label{lem:finalKummer}
Let $\mathcal{F}$ be a function field satisfying \eqref{con:subgroup}\eqref{con:genus}\eqref{con:rationalpoints}. Then the function field  
$\mathcal{F}$ is isomorphic to 
$\mathbb{F}_q(v, y)$, where the relation between $y$ and $v$ is given by the Kummer equation
\[  
y^{q-1} =\lambda \prod_{\alpha_i \in \mathbb{F}_q}\left(v-\alpha_i\right)^{n}=\lambda (v^q-v)^n,
\]
where $ \lambda \in \mathbb{F}_q^{*} $ and $ n $ is a positive integer coprime to $ q-1 $.  
\end{lem}
\begin{proof}
Let $  P_i, P_\infty, s_i, \lambda $ be the same notation as in Lemma \ref{lem:finalproof1}. 
We have already known that $\mathcal{F} $ is given by the Kummer equation
\[
    y^{q-1} = h(v) = \lambda \prod_{\alpha_i \in \mathbb{F}_q}\left(v-\alpha_i\right)^{s_i}, 
\]
and  $P_\infty$ is located at infinity.
 It suffices to show that all exponents $s_i$ are equal.  
Let $\rho$ be an element of $G$ such that  $\rho$ acts nontrivially on the orbit $\Omega_1$. By assumption, $\rho $ permutes the set $ \{ P_i \}$, and fixes $ P_\infty $. Write $\rho (P_i) = P_{\rho (i) }$ for each $i$. It is clear that \[ 
\sigma(\operatorname{div}(h))=\operatorname{div}(\sigma(h))
\]  for all $ h  \in \mathcal{F}$ and $\sigma \in G$. In particular,
we obtain 
\begin{align*}
    \operatorname{div}(\rho(h(v)))  = \rho(\operatorname{div}(h(v)))
    &=(q-1 )\Big( \sum_{i=0}^{q-1} s_i \rho(P_i)-(s_0+\dots+s_{q-1})\rho(P_{\infty}) \Big) \\
    &=(q-1 )\Big( \sum_{i=0}^{q-1} s_i P_{\rho(i)}-(s_0+\dots+s_{q-1})P_{\infty} \Big) . 
\end{align*}
Thus, 
\[
\rho(h(v))=\tilde{\lambda} \prod_{\alpha_i \in \mathbb{F}_q}\left(v-\alpha_i\right)^{s_{\rho^{-1} (i)}}
\]
for some $\tilde{\lambda}\in \mathbb{F}_q^*$. 
On the other hand, since $G$ is abelian, the automorphism $\rho$ commutes with the Galois group $H$. Applying Proposition \ref{prop:2.4},  we have $\rho(y)=f(v) y$ for some $f(v) \in \mathbb{F}_q(v)$ and 
\[
    f(v)^{q-1} h(v)=\rho(h(v)).
\]
This gives the relation
$$
f(v)^{q-1} \cdot \lambda \prod_{\alpha_i \in \mathbb{F}_q}\left(v-\alpha_i\right)^{s_i}=\tilde{\lambda} \prod_{\alpha_i \in \mathbb{F}_q}\left(v-\alpha_i\right)^{s_{\rho^{-1}(i)}}.
$$ 
Consequently, $\tilde{\lambda}=\lambda$ and
$$
f(v)^{q-1}=\prod_{\alpha_i \in \mathbb{F}_q}\left(v-\alpha_i\right)^{s_{\rho^{-1}(i)}-s_i}.
$$
Therefore, $q-1$ divides $s_{\rho^{-1}(i)}-s_i$ for each $i$. Since $\left| s_{\rho^{-1}(i)}-s_i \right| <q-1$, we can conclude that $s_{\rho^{-1}(i)}= s_i$. 
 Note that $G$ acts transitively on ${P_0,\dots,P_{q-1}}$. By choosing a suitable $\rho \in G$, we obtain the desired result.
\end{proof}



Now we prove the main result with the help of Lemma \ref{lem:finalKummer}. 

\begin{proof}[Proof of Theorem \ref{thm:main}]
From Lemma \ref{lem:finalKummer}, we have  
$$
y^{q-1}=h(v)=\lambda \prod_{\alpha_i \in \mathbb{F}_q}\left(v-\alpha_i\right)^{n}=\lambda (v^q-v)^n,
$$
where $n$ and $(q-1)$ are coprime. 
Choose two integers $a$, $b$ such that 
 $ a n+ b(q-1) =1 $ 
and set $ u:= y^a (v^q-v)^b \in \mathcal{F}$. Then 
\begin{equation}\label{eq:uv}
    u^{q-1}= ( y^a (v^q-v)^b)^{q-1}=  \lambda^a (v^q-v)^{an+b(q-1)} =\lambda^{a}  (v^q-v) .
\end{equation}
It is evident to see that 
\[
    u^n=y^{an}(v^q-v)^{bn}=y^{an+b(q-1)} \lambda^{-b}= \lambda^{-b} y.
\]
It yields that $y \in \mathbb{F}_q(v, u)$. Therefore, we get
 \[ \mathcal{F} \cong \mathbb{F}_q(v, y)=\mathbb{F}_q(v, u) . \]
  Since the expression in \eqref{eq:uv} coincides with the desired expression for $L\left(\Lambda_{x^2}\right)$ in Proposition \ref{kummerformoftheT^2}, we conclude that $\mathcal{F}$ is $\mathbb{F}_q$-isomorphic to $L\left(\Lambda_{x^2}\right)$.
\end{proof}

\section{Proof of Lemma \ref{lem:orbits}}\label{sec:proofofLemma}
    In this section, we aim to prove Lemma \ref{lem:orbits}. 
     Following the notation in Section \ref{sec:mainresult}, we let $ \mathcal{F} $ be the function field that satisfies the conditions \eqref{con:subgroup}\eqref{con:genus}\eqref{con:rationalpoints}.  
     
    \subsection{Extension of automorphisms} 
 Given a function field $\mathcal{F}/\mathbb{F}_q$, let $\newF = \overline{\mathbb{F}}_q \mathcal{F}$ be its constant field extension. By Zorn's Lemma, every automorphism of $\mathcal{F}$ over $\mathbb{F}_q$ extends to an automorphism of $\newF$ over $\overline{\mathbb{F}}_q$. More precisely, there is a natural injective group homomorphism
\[
    \Aut_{\mathbb{F}_q}(\mathcal{F}) \hookrightarrow \Aut_{\overline{\mathbb{F}}_q}(\newF),
\]
see \cite[Corollary 14.3.9]{MR2241963}.

 Consider two rational places $P_1$ and $P_2$ of $\mathcal{F}$, and suppose that some $\sigma \in \operatorname{Aut}_{\mathbb{F}_q}(\mathcal{F})$ satisfies $\sigma(P_1) = P_2$ (we denote the extension of $\sigma$ to $\overline{\mathbb{F}}_q \mathcal{F}$ by the same symbol). Let $P_i'$ be the unique place of $\overline{\mathbb{F}}_q \mathcal{F}$ lying over $P_i$ for $i = 1, 2$. Then we necessarily have $\sigma(P_1') = P_2'$.
Denote by $\newF^G$, $\newF^H$, $\newF^I$ the fixed fields of $\newF$ under the action of $G$, $H$, and $I$, respectively (see Notation \ref{not:subgroups}). Let $g^G$, $g^H$, and $g^I$ be the genera of $\newF^G$, $\newF^H$, and $\newF^I$, respectively.

\subsection{Short $G$-orbits}
Assume that $\tilde{\F}$ has exactly $k$ short $G$-orbits (see Definition \ref{def:shortOrbit}), denoted by $\Omega_1, \dots, \Omega_k$. For $i = 1, \dots, k$, let $l_i$ be the cardinality of $\Omega_i$. Then, by Equation \eqref{eq:GP} and the definition of short $G$-orbit, we have
\begin{align}\label{eq:l_i|q(q-1)}
    l_i | q(q-1) \quad \text{and }   l_i < q(q-1). 
\end{align} 
The following lemma yields that $ k\geqslant 2$.
\begin{lem}\label{lem:rational}
    The set of rational places of $\mathcal{F}$ (identified with the corresponding places of $\tilde{\mathcal{F}}$) is a union of short $G$-orbits. Moreover, there are at least two such short orbits.
\end{lem}
\begin{proof}
The set of rational places is stable under the action of $G$ because $G$ acts on $\mathcal{F}$ over $\mathbb{F}_q$ and preserves rationality. From Condition \eqref{con:rationalpoints}, $ \mathcal{F} $ contains exactly  $ q+1 $ rational places.
Let $ P $ be a rational place of $ \mathcal{F} $. The $G$-orbit $G(P)$ containing $P$ has cardinality $l$, and we have $l \leqslant q+1 < q(q-1)$ for any $q \geqslant 3$. The case $l = q+1$ is impossible, as $q+1 \nmid q(q-1)$. This implies that the set of rational places is the union of $k$ short $G$-orbits, with $k \geqslant 2$.
\end{proof}


\begin{remark}\label{rem:rationalplaces}
By Lemma~\ref{lem:rational}, the set of rational places of $\mathcal{F}$ decomposes into a disjoint union of short $G$-orbits. Let $\Omega_1$ and $\Omega_2$ be two distinct orbits among them, with cardinalities $l_1$ and $l_2$, respectively. Since the orbits are disjoint and consist entirely of rational places, we have
\[
l_1 + l_2 \leqslant q+1.
\]
In particular, each orbit is a proper nonempty subset of the set of rational places; hence
\[
0 < l_i < q+1 \qquad (i = 1,2).
\]
\end{remark}
 
\subsection{The genus of $\newF^G$}
Let $N = |G| = q(q-1)$. Denote by $g$ the genus of $\newF$.  
Applying the formula \eqref{eq:HurWitzgenusFormula12} to the Galois extension $\newF/\newF^G$ gives 
\begin{equation}\label{eq:2g2}
2g-2 \geqslant N(2g^G-2) + \sum_{i=1}^k (N - l_i).
\end{equation}
From Condition \eqref{con:genus}, we know $2g-2 = q(q-3)$. Substituting $N = q(q-1)$ into \eqref{eq:2g2}, we obtain
\begin{equation}\label{eq:2pq}
q(q-3) \geqslant q(q-1)(2g^G-2) + \sum_{i=1}^k \bigl(q(q-1)-l_i\bigr).
\end{equation}
From \eqref{eq:l_i|q(q-1)}, we have $l_i \leqslant \frac{q(q-1)}{2}$, and therefore 
\[ q(q-1)-l_i \geqslant \frac{q(q-1)}{2} \geqslant 0.  \] So the inequality \eqref{eq:2pq} reduces to the weaker estimate
\[
q(q-3) \geqslant  q(q-1)(2g^G-2).
\]
 Hence, $g^G \leqslant 1$. We now exclude the possibility $g^G = 1$. Substituting $g^G = 1$ into \eqref{eq:2pq}, we obtain
\[
q(q-3) \geqslant \sum_{i=1}^k \bigl(q(q-1)-l_i\bigr) \geqslant k \cdot \frac{q(q-1)}{2} .
\]
It follows that $k \leqslant 2\frac{q-3}{q-1} < 2$, which is a contradiction to Lemma~\ref{lem:rational}. Therefore, we conclude $g^G = 0$.

\subsection{Bounds for $k$: $2 \leqslant k \leqslant 3$}
Substituting $ g^G = 0 $ into the formula \eqref{eq:2pq} yields
\[
q(q-3) \geqslant -2q(q-1) + \sum_{i=1}^k \bigl(q(q-1)-l_i\bigr).
\]
It follows that
\begin{equation}\label{eq:simgai}
\sum_{i=1}^k l_i \geqslant (k-3)q(q-1) + 2q.
\end{equation} 
By Remark \ref{rem:rationalplaces}, we may assume that $\Omega_1$ and $\Omega_2$ are two distinct short orbits of rational places, which implies $l_1 + l_2 \leqslant q+1$. Moreover, it follows from \eqref{eq:l_i|q(q-1)} that  
\begin{align}\label{eq:upperbound}
    S_{\geqslant 3}: = \sum_{i=3}^k l_i \leqslant (k-2)\cdot\frac{q(q-1)}{2}.
\end{align} 
On the other hand, from \eqref{eq:simgai} we obtain a lower bound for $S_{\geqslant 3}$:
\begin{align}
S_{\geqslant 3}= \sum_{i=1}^k l_i - (l_1+l_2) & \geqslant (k-3)q(q-1) + 2q - (l_1+l_2) \nonumber \\
& \geqslant (k-3)q(q-1) + (q-1) . \label{eq:lowerbound}
\end{align} 
Combining \eqref{eq:upperbound} and \eqref{eq:lowerbound}, we get
\[
\frac{k-2}{2}\,q(q-1) \geqslant (k-3)q(q-1) + (q-1),
\]
which yields
\[
\frac{4-k}{2}\,q \geqslant 1.
\]
Thus, $k \leqslant 3$. Together with Lemma~\ref{lem:rational} which gives $k\geqslant 2$, we finally obtain
\[
2 \leqslant k \leqslant 3.
\]

        
       

          
 
\subsection{Analysis of the case $k=3$}
It suffices to exclude the possibility  $k=3$.
Assume now that there are exactly three short $G$-orbits, denoted by $\Omega_1,\Omega_2,\Omega_3$, with cardinalities $l_1,l_2,l_3$. From Lemma~\ref{lem:rational} we know that the set of rational places of $\tilde{\F}$ is a union of short $G$-orbits and contains at least two such orbits. Inequality \eqref{eq:simgai} (with $k=3$) gives
\begin{equation}\label{eq:l1l2l3}
l_1+l_2+l_3 \geqslant 2q.
\end{equation}
Since the total number of rational places is $q+1$, and $l_1+l_2+l_3 > q+1$ for any prime power $q\geqslant 3$, the three orbits cannot all consist entirely of rational places. Hence exactly two of them are composed of rational places. Without loss of generality, let $\Omega_1$ and $\Omega_2$ be the corresponding orbits. Then
\begin{align}\label{eq:l_1+l_2=q+1}
    l_1 + l_2 = q+1.
\end{align} 
Substituting this into \eqref{eq:l1l2l3} yields
\begin{align}\label{eq:l_3geqq-1}
    l_3 \geqslant q-1.
\end{align} 
By \eqref{eq:l_1+l_2=q+1}, we get $1 \leqslant l_1,l_2 \leqslant q$. 
Their sum is $q+1$, so either both lie strictly between $1$ and $q$, or one of them equals $1$ and the other equals $q$. Accordingly we distinguish two cases:
\begin{itemize}
    \item[(\Rmnum{1})] $1 < l_1 < q$ and $1 < l_2 < q$;\label{case:I}
    \item[(\Rmnum{2})] $\{l_1,l_2\} = \{1,q\}$.\label{case:II}
\end{itemize}
The two possibilities will be excluded separately.

        

The following lemma simplifies Case (\Rmnum{1}).
\begin{lem}\label{lem:p|l_1, l_2 is prime to p}
    Assume that $1<l_1<q$ and $1<l_2<q$. Then one of $l_1,l_2$ is coprime to $q$, and the other is divisible by $p$, where $q=p^t$ for some integer $t$ and prime number $p$.
\end{lem}

\begin{proof}
    By \eqref{eq:l_1+l_2=q+1}, reducing modulo $p$ shows that $p$ cannot divide both $l_1$ and $l_2$. Hence, it suffices to exclude the case in which $p$ divides neither $l_1$ nor $l_2$. Suppose then that $p \nmid l_1$ and $p \nmid l_2$. Under this assumption, both $l_1$ and $l_2$ are coprime to $q$. It follows from \eqref{eq:l_i|q(q-1)} that $l_i \mid (q-1)$.
    
    If both $l_1 < q-1$ and $l_2 < q-1$, then $l_i \mid (q-1)$ implies each $l_i \leqslant \frac{q-1}{2}$. Thus $l_1+l_2 \leqslant q-1$, contradicting $l_1+l_2 = q+1$. Hence the only possibility is $\{l_1,l_2\} = \{q-1,\,2\}$. 
    
    If \(q\) is even, then $2\nmid (q-1)$, contradicting the requirement that \(l_i \mid (q-1)\). 
    If $q$ is odd, now let $I$ be the subgroup of $G$ of order $q$. Since $q-1$ and $2$ are coprime to $q$, the orbit-stabilizer theorem implies that the stabilizer of any place in $\Omega_1$ and $\Omega_2$ has order divisible by $q$; consequently, $I$ is a subgroup of such a stabilizer. It follows that $I$ fixes every place in $\Omega_1 \cup \Omega_2$, and each place in these two orbits is totally ramified with ramification index $q$ in the extension $\newF/\newF^I$. Applying formula \eqref{eq:HurWitzgenusFormula12} to $\newF/\newF^I$ yields
    \[
        q(q-3) \geqslant (2g^I-2)q + (q-1)(q-1) + 2(q-1).
    \]
    Simplifying gives $g^I < 0$, a contradiction. Hence the assumption that both $l_1$ and $l_2$ are coprime to $p$ is impossible. Therefore exactly one of them is divisible by $p$.
\end{proof}
Without loss of generality, by Lemma \ref{lem:p|l_1, l_2 is prime to p}, Case (\Rmnum{1}) can be refined to   Case (\Rmnum{1}$'$) in which the cardinalities of the two short orbits of rational places are given by
\[
 l_1 = p^t h_1, l_2 = h_2 \text{ (or vice versa) with } h_1 \mid (q-1) , h_2 \mid (q-1)
\] 
and where $t$ is a positive integer satisfying $t < n$ (here $q = p^n$). 
By \eqref{eq:l_i|q(q-1)}, let $l_3 = p^l h_3$ where $h_3 \mid (q-1)$ and $l$ is an integer.

\subsection{Case (\Rmnum{1}$'$)} 
We now consider Case (\Rmnum{1}$'$). We need to compute the ramification indices for short orbits in the field extensions $\newF/\newF^I$ and $\newF^I/\newF^G $.
\begin{lem}\label{lem:theamificationindexofcase1}
Suppose that the cardinalities $ l_1 = p^{t} h_1, l_2 =  h_2, l_3 = p^{l} h_3 $, with $ h_i | (q-1) $ for each $i$ as before.
The ramification index $ e_i$ of  $\Omega_i$ in the field extension $\newF/\newF^I$ is given by 
\[
e_1 = p^{n-t},  e_2 = q,  e_3 =p^{n-l} .
\]
Accordingly, the ramification index $e_i'$ of any place lying under $\Omega_i$ in the field extension $\newF^I/\newF^G$ is given by 
\[
e_1' = \frac{q-1}{h_1} , e_2' = \frac{q-1}{h_2}, e_3' =\frac{q-1}{h_3} .
\]
\end{lem}

\begin{proof}
It follows from the orbit-stabilizer theorem that the stabilizer in $G$ of any place $P \in \Omega_1$ has order $p^{n-t} \cdot \frac{q-1}{h_1}$. Hence, this stabilizer contains a subgroup $I'$ of order $p^{n-t}$. As $G$ contains a unique subgroup of order $q$ (namely $I$), the subgroup $I'$ must lie in $I$; more precisely, $I \cap G_P = I'$. Consequently, we obtain $e_1 = |I'| = p^{n-t}$ by \eqref{eq:G_P}.

Denote the cardinality of the set of places of $\newF^I$ lying under the orbit $\Omega_1$ by $r$. From the Fundamental equality \eqref{eq:FundamentalEquality}, there are $p^t$ places lying over each place of $\newF^I$ under $\Omega_1$. Thus, $r=h_1$.  Again, the Fundamental equality \eqref{eq:FundamentalEquality} applied to $\newF^I / \newF^G$ yields $e_1'=\frac{q-1}{h_1}$. For the remaining orbits, the analysis follows a similar pattern.
\end{proof}
In Case (\Rmnum{1}$'$), we consider the cases \(l_3 = q-1\), \(l_3 = q\), and \(l_3 \geqslant q+1\) separately. For each possibility, we aim to obtain a contradiction using the conditions of Case (\Rmnum{1}$'$). 
 
If $l_3=q-1$, the ramification index $e_3$ of any place in $\Omega_3$ equals $q$ by Lemma \ref{lem:theamificationindexofcase1}.
       The formula \eqref{eq:HurWitzgenusFormula11} applied to $\newF/\newF^I$ gives 
       \begin{align*}
         2(1+\frac{q(q-3)}{2})-2 & \geqslant (2g^I-2)q+ e_1 l_1+e_2 l_2 +e_3 l_3 \\
         & = (2g^I-2)q+p^{n-t}(p^th_1)+qh_2+q(q-1).
       \end{align*} Thus,
       \[
           -\frac{(h_1+h_2)}{2}\geqslant g^I.
       \]
       This leads to a contradiction.
       
       If $l_3=q$,  the places in $\Omega_3$ are unramified in the field extension $\newF/\newF^I$ by Lemma \ref{lem:theamificationindexofcase1}.
        The formula \eqref{eq:HurWitzgenusFormula11} applied to $\newF/\newF^I$ gives 
      \begin{align*}
            2(1+\frac{q(q-3)}{2})-2 &\geqslant (2g^I-2)q+ e_1 l_1+e_2 l_2 +(e_3-1 )l_3 \\
            &= (2g^I-2)q+p^{n-t}(p^th_1)+qh_2.
      \end{align*}
       This implies
 \begin{align}\label{eq:11}
      g^I \leqslant \frac{q-1-(h_1+h_2)}{2}.
 \end{align}
But the Hurwitz Formula \eqref{eq:HurWitzgenusformula2} applied to $\newF^I/\newF^G$ gives 
       \[
         2g^I-2 = -2(q-1)+h_1(\frac{q-1}{h_1}-1)+h_2(\frac{q-1}{h_2}-1)+(q-1-1) .
       \]
       From this equation we can get $g^I=\frac{q-(h_1+h_2)}{2}$. This gives a contradiction to \eqref{eq:11}.

       If $l_3 \geqslant q+1$, we write $l_3 = p^l h_3$ as before, where $h_3 \mid (q-1)$. It is obvious that $h_3 \geqslant 2$, and $l\geqslant 1$. As before, by Lemma \ref{lem:theamificationindexofcase1} and the formula \eqref{eq:HurWitzgenusFormula11} applied to $\newF/\newF^I $ gives 
       \[
          2(1+\frac{q(q-3)}{2})-2  \geqslant (2g^I-2)q+qh_1+qh_2+qh_3.  
       \]Using this inequality, we derive
       \begin{align}\label{eq:12}
           g^I \leqslant \frac{q-1-(h_1+h_2+h_3)}{2}.
       \end{align}
        But the formula \eqref{eq:HurWitzgenusformula2} applied to $\newF^I/\newF^G$ gives 
       \[
         2g^I-2 = -2(q-1)+h_1(\frac{q-1}{h_1}-1)+h_2(\frac{q-1}{h_2}-1)+h_3(\frac{q-1}{h_3}-1).
       \]
       It follows that
       \[
           g^I=\frac{q+1-(h_1+h_2+h_3)}{2}.
       \]
       This gives a contradiction to \eqref{eq:12}. 
       So far, we have excluded the Case (\Rmnum{1}$'$).
       \subsection{Case (\Rmnum{2})} 
     We now turn to Case (\Rmnum{2}). Assume that $l_1 = q$ and $l_2 = 1$. As in Case (I), we first examine the corresponding ramification indices.
        \begin{lem}\label{lem:theramificationindexofcase2}
        Suppose that the cardinalities $ l_1 = q, l_2 =1, l_3 = p^{l} h_3 $, with $ h_3 | (q-1) $.
        The ramification index $ e_i$ of  $\Omega_i$ in the field extension $\newF/\newF^H$ is given by 
        \[
            e_1 = q-1,  e_2 = q-1,  e_3 =\frac{q-1}{h_3} .
        \]
        Accordingly, the ramification index $e_i'$ of any place lying under $\Omega_i$ in the field extension $\newF^H/\newF^G$ is given by 
        \[
            e_1' = 1,  e_2' = q,  e_3' =p^{n-l} 
        \]
     \end{lem}

\begin{proof}
    Since $|\Omega_1| = q$, the orbit-stabilizer theorem implies that the stabilizer of any place $P \in \Omega_1$ is precisely the subgroup $H \leqslant G$ of order $q-1$. Therefore, the claim for $\Omega_1$ follows directly. Since $l_2 = 1$, the stabilizer of the unique place in $\Omega_2$ is $G$. Hence, the result for $\Omega_2$ follows.
    For the orbit $\Omega_3$,
    since $l_3=p^l h_3$, the order of the stabilizer of any place $P$ in $\Omega_3$ under the action of $G$ is $p^{n-l} \frac{q-1}{h_3}$. Thus, the stabilizer of any place $P$ in $\Omega_3$ contains a subgroup $H'$ of order $\frac{q-1}{h_3}$. Since there exists only one subgroup of $G$ of order $q-1$, $H'$ must be contained in $H$. In particular, $H \cap G_P=H'$. It follows from \eqref{eq:G_P} that
         \[
             e_3=|H'|=\frac{q-1}{h_3}.
         \]
         Let $r$ denote the number of places of $\newF^H$ lying under the orbit $\Omega_3$. By the Fundamental Equality \eqref{eq:FundamentalEquality}, each such place has exactly $h_3$ places of $\Omega_3$ lying over it. Thus, $r=p^l$. The Fundamental equality \eqref{eq:FundamentalEquality} applied to $\newF^H / \newF^G$ yields $e_3'=p^{n-l}$. 
\end{proof}
 We consider the cases \(l_3 = q-1\), \(l_3 = q\), and \(l_3 \geqslant q+1\) separately as before. For each possibility, we aim to obtain a contradiction using the conditions of Case (\Rmnum{2}). 

If $l_3 = q-1$, then by Lemma \ref{lem:theramificationindexofcase2} the ramification indices of the short orbits $\Omega_i$ are known. Applying the Hurwitz Genus Formula \eqref{eq:HurWitzgenusformula2} to $\newF/\newF^H$ yields
\[
q(q-3) = (2g^H-2)(q-1)+(q-2)q + (q-2),
\]
which implies $g^H = 0$. On the other hand, applying formula \eqref{eq:HurWitzgenusFormula11} to the extension $\newF^H/\newF^G$ gives
\[
2g^H-2 \geqslant -2q + q + q,
\]
so that $g^H \geqslant 1$. This contradicts $g^H = 0$.

If $l_3 = q$, applying the Hurwitz Genus Formula \eqref{eq:HurWitzgenusformula2} and Lemma \ref{lem:theramificationindexofcase2} to the extension $\newF/\newF^H$ yields
\[
q(q-3) = (2g^H-2)(q-1) + q(q-2) + (q-2) + q(q-2),
\]
which simplifies to
\[
g^H = \frac{-q^2 + 2q}{2(q-1)}.
\]
For $q \geqslant 3$, the right-hand side is negative, contradicting the non-negativity of the genus.

       If $l_3=p^l h_3 \geqslant q+1$, we consider two cases depending on whether $h_3 = q-1$ or $h_3 < q-1$. If $h_3<q-1$, then the ramification index of any place in $\Omega_3$ is strictly greater than 1 since the intersection of $H$ and the stabilizer of $\Omega_3$ is nontrivial. Let $d_3$ denote the different exponent of any place in $\Omega_3$ in the field extension $\newF/\newF^H$ . By Theorem \ref{thm:Dedekind's Different Theorem}, $d_3 \geqslant 1$. The formula \eqref{eq:HurWitzgenusformula2} and Lemma \ref{lem:theramificationindexofcase2} applied to $\newF/\newF^H$ give 

       \begin{align*}
           q(q-3) &=(2g^H-2)(q-1)+(e_1-1)q+(e_2-1)+d_3 l_3 \\
           &\geqslant (2g^H-2)(q-1)+(q-2)q+(q-2)+1.
       \end{align*} 
      This inequality yields $g^H < 0$, a contradiction.
      
 If $h_3=q-1$, it follows that $l<n$ by \eqref{eq:l_i|q(q-1)}. By Lemma \ref{lem:theramificationindexofcase2}, we have $e_3 = 1$ and $e_3' = p^{n-l}$. Applying the Hurwitz Genus Formula \eqref{eq:HurWitzgenusformula2} to $\newF/\newF^H$ gives
       \[
           q(q-3)=(2g^H-2)(q-1)+(q-2)q+(q-2).
       \]This equality yields $g^H=0$.
       But formula \eqref{eq:HurWitzgenusFormula11} applied to $\newF^H/\newF^G$ gives 
       \[
           2g^H-2 \geqslant (0-2)q+(1-1)q+q+p^{n-l}p^l.
       \] This implies that $
           g^H \geqslant 1$.
       We arrive at a contradiction. Therefore, Case (\Rmnum{2}) is excluded.
       
       From the above discussion, there are precisely two short $G$-orbits.
 \subsection{The cardinalities of two short orbits}   
   From the discussion above, we have established that there are exactly two short $G$-orbits, with cardinalities $l_1$ and $l_2$.

Suppose first that $1 < l_1 < q$ and $1 < l_2 < q$. By Lemma \ref{lem:p|l_1, l_2 is prime to p}, one of $l_1, l_2$ is coprime to $q$ while the other is divisible by $p$. Without loss of generality, write $l_1 = p^t h_1$ and $l_2 = h_2$, where $h_1 \mid (q-1)$ and $h_2 \mid (q-1)$ as before. By Lemma \ref{lem:theamificationindexofcase1}, $e_1' = \frac{q-1}{h_1}$ and $e_2' = \frac{q-1}{h_2}$. Applying the Hurwitz Genus Formula \eqref{eq:HurWitzgenusformula2} to the extension $\newF^I/\newF^G$ yields
\[
2g^I - 2 = -2(q-1) + h_1\left(\frac{q-1}{h_1} - 1\right) + h_2\left(\frac{q-1}{h_2} - 1\right),
\]
which simplifies to
\[
g^I = \frac{2 - h_1 - h_2}{2}.
\]
Since $h_2 > 1$ and $h_1 \geqslant 1$, we obtain a contradiction. Therefore, this configuration cannot occur. Consequently, after possibly interchanging $l_1$ and $l_2$, we must have $l_1 = q$ and $l_2 = 1$.
 
In conclusion, the set of places of $\mathcal{F}/\mathbb{F}_q$ has exactly two short $G$-orbits under the action of $G$: $\Omega_1$ of cardinality $q$, consisting of $q$ $\mathbb{F}_q$-rational places, and $\Omega_2$ of cardinality $1$, consisting of a single $\mathbb{F}_q$-rational place.



 \bibliographystyle{amsplain}
 \bibliography{paper}

\end{document}